\documentclass[12pt,amsfonts]{article}
\usepackage{latexsym}
\usepackage{amsfonts,amssymb}

\newenvironment{proof}[1][Proof]{\noindent \textbf{#1.} }{$\Box$}
\usepackage{mathrsfs}
\usepackage{indentfirst}
\usepackage[dvipsnames,usenames]{color}
\usepackage{amsmath}
\usepackage{amssymb}
\usepackage{graphicx}
\usepackage{times}
\usepackage{graphics,color}
\newtheorem{theorem}{Theorem}[section]
\newtheorem{lemma}[theorem]{Lemma}
\newtheorem{remark}[theorem]{Remark}
\newtheorem{corollary}[theorem]{Corollary}
\newtheorem{proposition}[theorem]{Proposition}

\newtheorem{definition}[theorem]{Definition}

\newcommand{\bd}[1]{\begin{definition}\label{#1}\rm}
\newcommand{\ed}{\end{definition}}
\newcommand{\bt}[1]{\begin{theorem}\label{#1}}
\newcommand{\et}{\end{theorem}}
\newcommand{\bprop}[1]{\begin{proposition}\label{#1}}
\newcommand{\eprop}{\end{proposition}}
\newcommand{\bcor}[1]{\begin{corollary}\label{#1}}
\newcommand{\ecor}{\end{corollary}}

\begin{document}
\title{{A rotor configuration in $\mathbb{Z}^d$ where Schramm's bound of escape rates attains}}
\par
\author{Daiwei He\footnote{Email:dw.hefudan@gmail.com. Address:School of Mathematical Sciences, Fudan University, Shanghai, China}}
\date{}
\maketitle
{\footnotesize
\noindent
\begin{quote}
{\bf Abstract}
Rotor walk is deterministic counterpart of random walk on graphs. We study that under a certain initial configuration in $\mathbb{Z}^d$, $n$ particles perform rotor walks from the origin consecutively. They would stop if they hit the origin or $\infty$. When the dimension $d \geq 3$, the escape rate exists and it attains the upper bound of Oded Schramm \cite{10}. When the dimension $d=2$, the numbers of the particle escaping to $\infty$ are of order $n/\log n$. The limit of their quotient exist and also attains the upper bound of Florescu,Ganguly,Levine,Peres \cite{7} which equals to $\frac{\pi}{2}$. We use the results and the methods of the outer estimate for rotor-router aggregation in L.Levine and Y.Peres \cite{6}.
\noindent


\end{quote}
}
\section{Introduction}
Rotor walk is a deterministic counterpart of random walk on graphs. It was first introduced in Priezzhev at al. \cite{8}. Its intuitive definition is as follows. We arrange a fixed cyclical order of its neighbors to each vertex of the graph and a rotor pointing to some neighbor of each vetex. A particle starts from a vertex of the graph. It moves to the neighbor of the vertex where the particle currently locates following the direction of the rotor. And then the rotor of the vertex shifts to the next neighbor of the cyclical order. We mainly focus on the rotor walk on $\mathbb{Z}^d$, Here is a formal definition of rotor walk on $\mathbb{Z}^d$.

\begin{definition}
$\mathcal{E}=\{\pm{e_1},\pm{e_2},\dots,\pm{e_d}\}$ is the set of the $2d$ cardinal directions of $\mathbb{Z}^d$ and
$\mathcal{C}$ is the set of cyclical orders of $\mathcal{E}$. $m: \mathbb{Z}^d \rightarrow \mathcal{C}$ and rotor configuration $\rho$ maps
$\mathbb{Z}^d$ to $\mathcal{E}$. We call a sequence $x_0,x_1, \dots \subseteq \mathbb{Z}^d$ is a rotor walk of initial rotor configuration $\rho$ if
there exists rotor configuration $\rho=\rho_0,\rho_1, \dots$ such that for all $n \geq 0$ \[x_{n+1}=x_n+\rho_n(x_n)\]
and \[\rho_{n+1}(x_n)=m(x_n)(\rho_n(x_n))\] and for $x\neq x_n$, $\rho_{n+1}=\rho_n$ where $m(x_n)$ is recognized as the permutation the cyclical order corresponds to.
\end{definition}

In our paper we assume $\forall x\in \mathbb{Z}^d$,$m(x)$ is independent of $x$. We denote $m(x)$ to be $m$.

In $\mathbb{Z}^d$ and the initial rotor configuration is $\rho$. A particle q perform rotor walk starting from the origin 0. There are two possible situations:

1. q return 0 eventually.

2. q does not return 0 and for all sites in $\mathbb{Z}^d$, q visits them only finite times.\\
As with the second situation, for $x \in \mathbb{Z}^d$, denote $d(x,0)$ to be the graph distance from $x$ to 0, namely the minimum number of edges of the path from $x$ to 0. We know that if q visited $x$ $(2d)^{d(x,0)}$ times, it must visit 0. Hence for all points in $\mathbb{Z}^d$ q visits them finite times and q would escape to infinity.

The particles in turn perform rotor walk from 0 means that the first particle performs rotor walk starting from 0 until meeting some stopping requirements(for example, hitting $\{0\} \cup \{\infty\}$) and the current rotor configuration is different from the initial configuration. Regarding the current configuration as the initial configuration, the second particle performs rotor walk from 0 until meeting some stopping requiements. The third particle's initial configuration is the configuration after the second particle finishes its rotor walk. Then the process goes on following the above rules.

If $n$ particles in turn perform rotor walk until either hitting 0 or escaping to infinity, denote the number of the particles escaping to infinity to be $I(\rho,n)$.

To measure the intensity of transience and recurrence of the initial configuration, consider the behavior of $I(\rho,n)/n$ when $n$ tends to infinity. Schramm \cite{10} proved for any initial configuration $\rho$,

\[\limsup_{n\to\infty}\frac{I(\rho,n)}{n}\leq \alpha_d\]
where $\alpha_d$ is the escaping probability of $d$-dimensional random walk.

Although the upper bound of the upper limit of $I(\rho,n)/n$ does not depend on initial configuration, the lower limit of $I(\rho,n)/n$ depends on initial configuration. In Omer Angel, Alexander E.Holroyd \cite{2}, they proved $\forall d \geq 2$, there exists an initial rotor configuration $\overline{ \rho}$ such that $I(\overline{ \rho},n) \equiv 0$. The method was introduced in Tulasi Ram Reddy A \cite{12}. Hence we know
\[\liminf_{n \to \infty}\frac{I(\overline{ \rho},n)}{n}=0\]

However, in Florescu,Ganguly,Levine,Peres \cite{7}, let $\tilde{\rho}(x) \equiv e_d$. \\
When $d=2$, for any initial configuration $\rho$,
\[\limsup_{n\to\infty}\frac{I(\rho,n)}{n/\log n}\leq \frac{\pi}{2},\quad \liminf_{n\to\infty}\frac{I(\tilde{\rho},n)}{n/\log n}>0.\]
When $d\geq3$,  \[\liminf_{n\to\infty}\frac{I(\tilde{\rho},n)}{n}>0.\]

A problem is that whether there exists an initial configuration $\rho^\prime$ in $\mathbb{Z}^d$ such that\\
when $d=2$, \[\lim_{n\to\infty}\frac{I(\rho^\prime,n)}{n/\log n}=\frac{\pi}{2}\]
and when $d\geq3$, \[\lim_{n\to\infty}\frac{I(\rho^\prime,n)}{n}=\alpha_d.\]

The definition of rotor walk on graphs is similar with rotor walk in $\mathbb{Z}^d$. For rotor walk on trees, Omer Angel,Alexander E. Holroyd \cite{1} gave a good answer to the above question. If $n$ particles in turn perform rotor walk from the root 0 of the tree $T$ until either returning 0 or escaping to infinity. For an initial configuration $\rho$ satisfying only finite number of vertices' initial rotor point to the root 0,
\[\lim_{n \to \infty}\frac{I(\rho, n)}{n}=\alpha\]
where $\alpha$ is the escaping probability of simple random walk on $T$.

In this paper we will find a rotor configuration attaining the upper bound of Schramm \cite{10} when $d \geq 3$ and attaining the upper bound $\frac{\pi}{2}$ in Florescu,Ganguly,Levine,Peres \cite{7} when $d=2$. In the following proof, denote $\rho_0(x)=+e_d$ if $x_d\geq0$ while $\rho_0(x)=-e_d$ if $x_d<0$ where $e_d$ is the $d$th-dimensional coordinate of $x$.

Our proof depends on an assumption of the cyclical order $m$. We know for any $e \in \mathcal{E}$, there exists $k$, $ 0 \leq k \leq 2d-1$ such that $m^{(k)}(e_d)=e$. Define a map $\eta: \mathcal{E} \rightarrow \{0,1,\dots,2d-1\}$ such that $\eta(e)=k$. Our assumption is that
\begin{equation}
\exists i, 1\leq i \leq d-1, \hbox{ such that } (\eta(e_i)-\eta(-e_d))(\eta(-e_i)-\eta(-e_d))<0.
\end{equation}
Intuitively, it means that $e_i$ and $-e_i$ could separate $e_d$ and $-e_d$ in the cyclical order $m$. For example, in $\mathbb{Z}^2$ the counterclockwise and clockwise rotation, and in $\mathbb{Z}^d$ where $d \geq 3$ counterclockwise and clockwise rotation after projecting the $2d$ directions onto a suitable 2-dimensional plane both satisfies the above assumption of $m$. Moreover, without loss of generality, in the following proof we assume that the direction satisfies (1) is $e_{d-1}$ unless other case specifically mentioned. Our first result is
\begin{theorem}
When $d=2$, \[\lim_{n\to\infty}\frac{I(\rho_0,n)}{n/\log n}=\frac{\pi}{2}.\]
\end{theorem}

The $d \geq 3$ is more complicate. We use the method and idea of rotor-router aggregation in L.Levine,Y.Peres \cite{6}. In $\mathbb{Z}^d$ $n$ particles in turn perform rotor walk starting from 0 until stepping onto a site that has never been visited by the previous particles. The process is called rotor-router aggregation. When rotor-router aggregation finishes denote the set of the sites occupied by particles to be $A_n$. The same with L.Levine,Y.Peres \cite{6}, denote $n=\omega_d r^d$ where $\omega_d$ is the volume of $d$ dimensional ball.

Using the abelian property Lemma 2.4 and rotor-router aggregation we obtain
\begin{theorem}
When $d\geq3$, \[\lim_{n\to\infty}\frac{I(\rho_0,n)}{n}=\alpha_d\]
\end{theorem}

\section{2-dimensional case}

In this section we will prove the 2-dimensional case.

For $A \subseteq \mathbb{Z}^d$, $\partial A: =\{y \in A^c: \exists x \in A, s.t. x \sim y\}$, $S_r:=\{x \in \mathbb{Z}^d: r \leq |x| <r+1\}, B_r:=\{x \in \mathbb{Z}^d: |x|<r \}$. We follow the idea of Florescu,Ganguly,Levine,Peres [7] by using another different experiment.

When the initial configuration is $\rho$, $n$ particles in turn perform rotor walk starting from the origin 0 until hitting $ \partial B_r$, denote the times the $n$ particles leaving the site $x$ to be $u_n^r(x)$; When the initial configuration is $\rho$, $n$ particles in turn perform rotor walk starting from the origin 0 until escaping to infinity, denote the times the $n$ particles leaving the site $x$ to be $u_n(x)$.

When the initial configuration is $\rho$, $u_n(0)$ particles in turn perform rotor walk starting from the origin 0 until either returning to 0 or escaping to infinity. Because an excursion from 0 to 0 in the trajectory of a particle which stops once escaping to infinity could be regarded as the trajectory of another particle which stops once either escaping to infinity or returning 0. The above process is the same as we letting $n$ particles in turn perform rotor walk until escaping to infinity. So by definition of $I(\rho,n)$, we know $I(\rho, u_n(0))=n$. Moreover, based on the above reason, we have when $0 \leq k \leq u_{n+1}(0)-u_n(0)-1$, $I(\rho,u_n(0)+k)=n$.

We also note that for initial configuration such that
\[\lim_{k\to\infty} I(\rho,k)=\infty\]
$u_n(x)$ is well-defined for all $n \geq 1$. And obviously $\lim_{k\to\infty} I(\rho_0,k)=\infty$. The next lemma is about the way the particles goes to infinity if we perform rotor walk in $\mathbb{Z}^d$ when the initial configuration is $\rho_0$.

\begin{lemma}
When $d\geq2$ and the initial configuration is $\rho_0$, the particles in turn perform rotor walk starting from the origin 0 until escaping to infinity. Then the only for the particle to escape to infinity is to follow either $+e_d$ or $-e_d$ after finite steps.
\end{lemma}

\begin{proof}
The first particle escapes to infinity following $+e_d$.

If the first $n$ particles escaping to infinity follow either $+e_d$ or $-e_d$ after finite steps. For $r \in \mathbb{Z}$,
\[H_{d-1}(r):=\{(x_1,\dots, x_{d-1},r):(x_1,\dots,x_{d-1}) \in \mathbb{Z}^{d-1}\}.\]

When n particles in turn perform rotor walk until escaping to infinity, denote
\[ P_n(r):=\begin{cases}
\{(x_1,\dots,x_{d-1},r): \exists x_d \geq 0, \hbox{ such that } u_n(x_1,\dots,x_d)>0\} & r\geq 0\\
\{(x_1,\dots,x_{d-1},r): \exists x_d < 0, \hbox{ such that } u_n(x_1,\dots,x_d)>0\} & r< 0
\end{cases} \]
and
\[h_n^+:=\min \{h \geq 0: \rho_n(x)=m(e_d),\forall x \in \bigcup_{r>h}P_n(r)\}\]
and
\[h_n^-:=\min\{h \geq 0: \rho_n(x)=m(-e_d),\forall x \in \bigcup_{r<-h}P_n(r)\}\]
where $\rho_n(x)$ represents the rotor configuration of $x$ after $n$ particles escape to infinity. The right sides of the definition of $h_n^+$ and $h_n^-$ are not null because of the assumption for the previous $n$ particles. Thus these definitions are well-defined.

Then for the $(n+1)$th particle escaping to infinity£¬the particle must hit
\[H_{d-1}(h_n^++1)\bigcup H_{d-1}(-h_n^--1)\bigcup (\partial (\bigcup_{-h_n^- \leq r \leq h_n^+}P_n(r))).\]

Obviously,the sites in $(\bigcup_{r \in \mathbb{Z}}P_n(r))^c$ have never been visited by the first $n$ particles. The rotor configuration of these sites are the same as their initial configuration.

If the $(n+1)$th particle hit $H_{d-1}(h_n^++1)$, the particle would follow $m(e_d)$ until $\partial (\bigcup_{r \in \mathbb{Z}}P_n(r))$ and then it would follow $e_d$ until $\infty$;
If the $(n+1)$th particle hit $H_{d-1}(-h_n^--1)$, the particle would follow$m(-e_d)$ until $\partial (\bigcup_{r \in \mathbb{Z}}P_n(r))$ and then it would follow $-e_d$ until $\infty$;
If the $(n+1)$th particle hit $\partial (\bigcup_{-h_n^- \leq r \leq h_n^+}P_n(r)))\bigcap \{x:x_d \geq 0\}$, the particle would follow $e_d$ until $\infty$;
If the $(n+1)$th particle hit $\partial (\bigcup_{-h_n^- \leq r \leq h_n^+}P_n(r)))\bigcap \{x:x_d < 0\}$, the particle would follow $-e_d$ until $\infty$.

Thus the $(n+1)$th particle would follow either $e_d$ or $-e_d$ after finte steps.
\end{proof}

We make more remarks about the definition of $h_n^+$ and $h_n^-$ in the above proof. Actually, $h_n^+$ is the maximal $d$th-dimensional coordinate of the sites visited by the first $n$ particles at least twice and $-h_n^-$ is the minimum $d$th-dimensional coordinate of the sites visited by the first $n$ particles at least twice. First there exists $x \in P_n(h_n^+)$ such that $\rho_n(x) \neq m(e_d)$. The first $n$ particles visit $x$ at least twice. Also, if the first $n$ particles visit $y \in P_n(h_n^++1)$ at least twice, and as $\rho_0(y)=e_d$ and $\rho_n(y)=m(e_d)$, we know the first $n$ particles pass through edge $(y, y+e_d)$ at least twice. Hence they visit $y+e_d$ at least twice. The same method could be used to obtain the first $n$ particles visit every site of the lattice line $\{z:z=y+ke_d, k\in \mathbb{N}\}$ at least twice. This is contradictory to the escaping structures we proved in Lemma 2.1. Thus $h_n^+$ is the maximal $d$th-dimensional coordinate of the sites visited by the first $n$ particles at least twice. The similar conclusion is valid for $h_n^-$.
In the following arguments $H_{d-1}(k),P_n(r), h_n^+, h_n^-$ are the same meanings with the above proof.

\begin{lemma}
When $d \geq 2$, we have $h_n^+ \leq n$, $h_n^-\leq n$.
\end{lemma}

\begin{proof}
We only need to prove $\forall k \in \mathbb{N}, h_{k+1}^+-h_k^+ \leq 1$.

After the $k$th particles escape to $\infty$, the first particle leading to twice visits on a site of the the hyperplane $\{x: x_d=h_k^++1\}$ must follow $m(e_d)$ until hitting $\partial (\bigcup_{m \in \mathbb{Z}} P_n(m))$. Then it would follow $e_d$ until $\infty$. After this particle finishes its rotor walk, the number, $l$,  of the particles escaping to infinity must be no less than $k+1$. So we have
\[h_{k+1}^+ \leq h_l^+= h_k^++1.\]
The above equality uses the monotonicity of $h_n^+$ depending on $n$. We could also know $h_n^-\leq n$ using the same method.
\end{proof}

Notice that we use the assumption (1) for $m$ in the above two proofs. Because $m(e_d) \neq -e_d$, the particle could reach $\partial (\bigcup_{r \in \mathbb{Z}}P_n(r))$ through $m(e_d)$. For example in $\mathbb{Z}^2$ the only permissible cyclical orders are north$\rightarrow$east$\rightarrow$south$\rightarrow$west$\rightarrow$north and north$\rightarrow$west$\rightarrow$south$\rightarrow$east$\rightarrow$north.

\begin{lemma}
For initial configuration $\rho$ such that $\lim_{n\to\infty} I(\rho,n)=\infty$, we have\\
When $d\geq3$, \[\limsup_{n\to\infty}\frac{n}{u_n(0)}=\limsup_{n\to\infty}\frac{I(\rho,n)}{n},\quad \liminf_{n\to\infty}\frac{n}{u_n(0)}=\liminf_{n\to\infty}\frac{I(\rho,n)}{n}\]
When $d=2$, \[\limsup_{n\to\infty}\frac{n}{u_n(0)/\log u_n(0)}=\limsup_{n\to\infty}\frac{I(\rho,n)}{n/\log n},\quad \liminf_{n\to\infty}\frac{n}{u_n(0)/\log u_n(0)}=\liminf_{n\to\infty}\frac{I(\rho,n)}{n/\log n}.\]
\end{lemma}

\begin{proof}
Obviously when $0 \leq k \leq u_{n+1}(0)-u_n(0)-1$, $I(\rho,u_n(0)+k)=n$.

So when $u_k(0) \leq n <u_{k+1}(0)$,
\[\frac{I(\rho,n)}{n}=\frac{I(\rho,u_k(0))}{n}=\frac{I(\rho,u_k(0))}{u_k(0)}. \frac{u_k(0)}{n} \leq \frac{I(\rho,u_k(0))}{u_k(0)}\]
\[\frac{I(\rho,n)}{n}=\frac{I(\rho,u_{k+1}(0))-1}{u_{k+1}(0)}.\frac{u_{k+1}(0)}{n} \geq \frac{I(\rho,u_{k+1}(0))-1}{u_{k+1}(0)}\]
Thus we could draw the above conclusion when $d \geq 3$. The same method could be used to prove the case when $d=2$.
\end{proof}

Hence we only need to prove that if the  initial configuration is $\rho_0$, when $d=2$,
\[\lim_{n\to\infty}\frac{n}{u_n(0)/\log u_n(0)}=\frac{\pi}{2}\]
and when $d\geq 3$,
\[\lim_{n\to\infty}\frac{n}{u_n(0)}=\alpha_d.\]

First of all, we state the following abelian property of rotor walk without proof. This property is proved in \cite{3} and also mentioned in L.Levine,Y.Peres \cite{6}, Florescu,Ganguly,Levine,Peres \cite{7}, Alexander E. Holroyd, L.Levine \cite{3}. Abelian property says that the position of the particles and the times the particles exit from certain site when rotor walk finishes do not depend on the choice we choose the particles in the roter-router process.

\begin{lemma}
(Abelian property) For a finite graph $\Gamma =(V,E), W \subseteq V$, on every vertex of $\Gamma$ there exists some particles. If there is a initial configuration $\rho$ on the graph, each step we choose a particle on $V \backslash W$ and perform one step rotor walk. When the particles hit $W$, they would stop. In the end all the particles stay on $W$. Then the final position of the particles and the times the particles exit from each site of the graph $\Gamma$ when rotor walk finishes do not depend on the choice we choose the particles in the rotor-router process.
\end{lemma}

Next we begin to prove the 2-dimensional case.
Let $f:\mathbb{Z}^d \rightarrow \mathbb{R}$. For $x\sim y$,
\[\nabla f(x,y):=f(y)-f(x).\]
We denote $\mathbb{L}^d=(\mathbb{Z}^d,\mathbb{E}^d)$ to be the graph consisting of $d$ dimension Euclidian sites and their incident edges. The edge of the graph is denoted by the ends of the edge. Denote
\[g:\mathbb{E}^d \rightarrow \mathbb{R}, \textrm{div} g(x):=\frac{1}{2d}\sum_{y\sim x}g(x,y)\]
and
\[\triangle f(x):=\textrm{div} (\nabla f)(x)=\frac{1}{2d}\sum_{y\sim x}f(y)-f(x).\]

The next lemma is from L.Levine,Y.Peres \cite{6}.
\begin{lemma}
$n$ particles in turn perform rotor walk starting from the origin 0 until hitting $\partial B_r$. $\forall (x,y) \in \mathbb{E}^d$, we denote $N_n(x,y)$ to be the times these particles go through the edge $(x,y)$.
\[K_n(x,y):=N_n(x,y)-N_n(y,x)\]
Then there exists $R_n: \mathbb{E}^d \rightarrow \mathbb{Z}$, such that $|R_n(x,y)| \leq 4d-2 $, and also
\[\nabla u_n^r(x,y)=-2dK_n(x,y)+R_n(x,y)\] for all edges $(x,y) \in \mathbb{E}^d$.
\end{lemma}

Let $(X_i)_{i\geq 0}$ be simple random walk on $\mathbb{Z}^d$. For $x,y \in B_r$, $T:=\min \{t>0 :X_t \in \partial B_r\}$, $G_r(x,y):=\mathbb{E}_x\# \{i<T: X_i=y\}$. We cite the following results about classical potential theory of random walk from G. F. Lawler \cite{5}.
When $x\neq 0$,
\[ G_r(x,0)=\begin{cases}
a_d(|x|^{2-d}-r^{2-d})+O(|x|^{1-d}) & d\geq 3\\
\frac{\pi}{2}(\log r-\log |x|)+O(|x|^{-1}) & d=2
\end{cases} .\]
When $d=2$,
\[G_r(0,0)=\frac{2}{\pi}\log r+O(1).\]

The next lemma comes from L.Levine,Y.Peres \cite{6}.

\begin{lemma}

There exists a constant $C$ depending only on dimension $d$, $\forall x \in B_r, \forall \rho$ where $0< \rho \leq r$, such that
\[\sum_{y \in B_r, |x-y| \leq \rho}\sum_{z\sim y}|G_r(x,y)-G_r(x,z)| \leq C\rho \]
\end{lemma}

We know for $x \in B_r$

\[\triangle u_n^r(x)=-2d(\textrm{div} K_n)(x)+\textrm{div} R_n(X)=-1_{\{x=0\}}n+\textrm{div} R_n(x)\]
and
\[\triangle G_r(x,0)=-1_{\{x=0\}}.\]
So
\[\triangle (u_n^r(x)-nG_r(x,0))=\textrm{div} R_n(x)\]

Next an estimation between $\textrm{div} R_n(x)$ and $|u_n^r(0)-nG_r(0,0)|$ is expected to be given and L.Florescu,S.Ganguly,L.Levine,Y.Peres \cite{7} gave one way to do this. Their final conclusion is that the lower limit of escape rates is larger than 0 while our conclusion is that the limit of escape rates exists and equals to the upper bound of the upper limit. For a self-contained reason we give a relatively complete reasoning. This method is from L.Florescu,S.Ganguly,L.Levine,Y.Peres \cite{7}.

First,
\[u_n^r(x)=\sum_{k\geq 0}\mathbb{E}_x(u_n^r(X_{k\wedge T})-u_n^r(X_{{(k+1)}\wedge T})).\]
Also
\[\mathbb{E}_x(u_n^r(X_{k\wedge T})-u_n^r(X_{{(k+1)}\wedge T})|\mathcal{F}_{k\wedge T})=-\bigtriangleup u_n^r(X_k)1_{\{k<T\}}.\]
So

\begin{align*}
u_n^r(x) & =\sum_{k\geq 0}\mathbb{E}_x[-\bigtriangleup u_n^r(X_k)1_{\{k<T\}}]\\
   & =\sum_{k \geq 0}\mathbb{E}_x[n1_{\{X_k=0,k<T\}}-\textrm{div} R_n(X_k)1_{\{k<T\}}]\\
   & =nG_r(x,0)-\sum_{k \geq 0}\mathbb{E}_x[1_{\{k<T\}}\textrm{div} R_n(X_k)]
\end{align*}
Thus
\[u_n^r(x)-nG_r(x,0)=-\frac{1}{2d}\sum_{k \geq 0}\mathbb{E}_x[1_{\{k<T\}}\sum_{z\sim X_k}R_n(X_k,z)].\]
Denote $N(x)$ to be the number of the edges connect $x$ with $\partial B_r$. Because $|R_n| \leq 4d-2$,

\begin{align*}
|u_n^r(x)-nG_r(x,0)| & \leq \frac{1}{2d}|\sum_{y,z \in B_r,y\sim z}[G_r(y,x)R_n(y,z)]|+2|\sum_{k \geq 0}\mathbb{E}_x[1_{\{k<T\}}N(X_k)]|\\
& \leq \frac{1}{2d}|\sum_{y,z \in B_r,y\sim z}[G_r(y,x)R_n(y,z)]|+C_1.
\end{align*}
The reason of the last inequality is that 2$\sum_{k \geq 0}\mathbb{E}_x[1_{\{k<T\}}N(X_k)]<C_1$ for a constant $C_1$ depending only on dimension $d$. By the definition of $R_n$, $R_n(y,z)=-R_n(z,y)$. So
\[|u_n^r(x)-nG_r(x,0)| \leq \frac{1}{4d}|\sum_{y,z \in B_r,y\sim z}[(G_r(y,x)-G_r(z,x))R_n(y,z)]| +C_1.\]

Next we give a proof of Theorem 1.1.

\begin{proof}
When $d=2$, because $\forall x=(x_1, x_2) \in P_n(0)$, there exist a path from 0 to some site $\overline{x}$ on $l_x:=\{y \in \mathbb{Z}^2:y=x+ke_2 ,k \in \mathbb{N}\}$ such that every site on the path belongs to $\bigcup_{-h_n^- \leq r \leq h_n^+}P_n(r)$. These sites in $\{x=(x_1, x_2) \in \mathbb{Z}^2: x_2 \geq 0\}$ could be projected onto $H_1(0)$ while these sites in $\{x=(x_1, x_2) \in \mathbb{Z}^2: x_2 < 0\}$ could be projected onto $H_1(-1)$. Thus we could find a path located in $P_n(0) \cup P_n(-1)$ and connect 0 with $x$. Every point on this new path correspond to a particle escape to infinity and theses particles are obviously different from each other. So $|x_1| \leq n$. The same method could be use to prove if $x \in P_n(-1)$, $|x_1| \leq n$.

Also by Lemma 2.2, there exists a constant $C_2$ such that
\[\bigcup_{-h_n^- \leq r \leq h_n^+}P_n(r)\subseteq B_{C_2n}.\]
Let $r=C_2n$. Obivously $u_n(0)=u_n^{C_2n}(0)$. So
\[|u_n(0)-n(\frac{2}{\pi}\log n+O(1))| \leq \sum_{y,z \in B_{C_2n},y\sim z}|G_{C_2n}(y,0)-G_{C_2n}(z,0)|+C_1.\]
Some simple calculus could lead to that when $s>t>0$£¬
\[|\frac{s}{\log s}-\frac{t}{\log t}| \leq \frac{1}{\log t}|s-t|.\]
Thus
\begin{align*}
|\frac{u_n(0)}{\log u_n(0)}-\frac{n(\frac{2}{\pi}\log n+O(1))}{\log [n(\frac{2}{\pi}\log n+O(1))]}| & \leq \frac{\sum_{y,z \in B_{C_2n},y\sim z}|G_{C_2n}(y,0)-G_{C_2n}(z,0)|+C_1}{\log (u_n(0)\wedge n(\frac{2}{\pi}\log n+O(1)))}.
\end{align*}
Divided by $n$ on both sides of the inequality and by Lemma 2.6, let $n \rightarrow \infty$. We could know
\[\limsup_{n \to \infty}|\frac{u_n(0)}{n\log u_n(0)}-\frac{2}{\pi}| \leq \limsup_{n \to \infty}\frac{1}{\log (u_n(0)\wedge n(\frac{2}{\pi}\log n+O(1)))}[CC_2+\frac{C_1}{n}]=0.\]
So
\[\lim_{n \to \infty}\frac{n}{u_n(0)/\log u_n(0)}=\frac{\pi}{2}.\]
By Lemma 2.3, we obtain
\[\lim_{n \to\infty}\frac{I(\rho_0,n)}{n/\log n}=\frac{\pi}{2}.\]
\end{proof}

\section{Higher dimensional case}

\begin{lemma}
When $d\geq3$, there exists $R(n)>0$ such that
\[\bigcup_{-h_n^- \leq r \leq h_n^+}P_n(r)\subseteq B_{R(n)}\]
and
\[\lim_{n \to \infty}\frac{R(n)}{n}=0.\]
\end{lemma}
The proof of the above lemma will be left to the last section.

The next lemma comes from L.Florescu,S.Ganguly,L.Levine,Y.Peres \cite{7}¡£
\begin{lemma}
When $d \geq 3$, there exists a sufficient small constant $\beta$ depending only on dimenstion $d$ such that for any initial configuration, $\forall x \in \partial B_{\beta n^{\frac{1}{d-1}}}$, $u_n^{n^{\frac{1}{d-1}}}(x)>0$.
\end{lemma}

Next we give the proof of Theorem 1.2.

\begin{proof}
When $d \geq 3$, let $r>R(n)$.
Because
\[u_n^r(x)-nG_r(x,0)=-\frac{1}{2d}\sum_{y \in B_r}\sum_{z \sim y}G_r(x,y)R_n(y,z).\]
By Lemma 2.6, we obtain
\begin{align*}
|u_n^r(0)-nG_r(0,0)| & \leq \sum_{y,z \in B_{R(n)},y\sim z}|G_r(y,0)-G_r(z,0)|+ \\
& \frac{1}{2d}|\sum_{y \in B_r\backslash B_{R(n)},y\sim z}G_r(y,0)R_n(y,z)|\\
& \leq CR(n)+\frac{1}{2d}|\sum_{y \in B_r\backslash B_{R(n)},y\sim z}G_r(y,0)R_n(y,z)|.\\
\end{align*}
And we know
\begin{align*}
\sum_{y \in B_r\backslash B_{R(n)},y\sim z}G_r(y,0)R_n(y,z) & =\sum_{y \in B_r\backslash B_{R(n)}}G_r(y,0)(\sum_{z\sim y }R_n(y,z))\\
& =\sum_{y \in B_r\backslash B_{R(n)}}G_r(y,0)(2d\triangle u_n^r(y)).
\end{align*}

Denote $F=\{y \in B_r\backslash B_{R(n)}: \forall z \in (B_r\backslash B_{R(n)})^c,  z \nsim y\}$.

When $y\in \bigcup_{m \in \mathbb{Z}}P_n(m) \bigcap F $ and $\forall z \in (\bigcup_{m \in \mathbb{Z}}P_n(m))^c \bigcap (B_r\backslash B_{R(n)}), y \nsim z$, $u_n^r(z)=u_n^r(y)=1$. So $\triangle u_n^r(y)=0$.

When $y\in (\bigcup_{m \in \mathbb{Z}}P_n(m))^c \bigcap F $ and $\forall z \in (\bigcup_{m \in \mathbb{Z}}P_n(m)) \bigcap (B_r\backslash B_{R(n)}),y \nsim z$, $u_n^r(z)=u_n^r(y)=0$. So $\triangle u_n^r(y)=0$.

When $y\in \bigcup_{m \in \mathbb{Z}}P_n(m) \bigcap F $ and $\exists z \in (\bigcup_{m \in \mathbb{Z}}P_n(m))^c \bigcap (B_r\backslash B_{R(n)}) $ such that $y \sim z$. Denote $M(y)$ to be the number of the sites in $(\bigcup_{m \in \mathbb{Z}}P_n(m))^c \bigcap (B_r\backslash B_{R(n)})$ connecting $y$. Under this condition $2d\triangle u_n^r(y)=-M(y)$.

When $y\in (\bigcup_{m \in \mathbb{Z}}P_n(m))^c \bigcap F $ and $\exists z \in (\bigcup_{m \in \mathbb{Z}}P_n(m)) \bigcap (B_r\backslash B_{R(n)})  $ such that $y \sim z$. Denote $W(y)$ to be the number of the sites in $(\bigcup_{m \in \mathbb{Z}}P_n(m)) \bigcap (B_r\backslash B_{R(n)})$ connecting $y$. Under this condition $2d\triangle u_n^r(y)=W(y)$.

Due to the four situations above,
\begin{align*}
& \sum_{y\in (\partial (\bigcup_{m \in \mathbb{Z}}P_n(m))) \bigcap (B_r\backslash B_{R(n)}),x \in (\bigcup_{m \in \mathbb{Z}}P_n(m)) \bigcap (B_r\backslash B_{R(n)}),x \sim y}(G_r(y,0)-G_r(x,0))\\
& =\sum_{y \in F}G_r(y,0)(2d\triangle u_n^r(y))
 +\sum_{y \in (\partial (\bigcup_{m \in \mathbb{Z}}P_n(m)))\bigcap ((B_r\backslash B_{R(n)}) \bigcap F^c)}W(y)G_r(y,0)\\
& -\sum_{z \in (\bigcup_{m \in \mathbb{Z}}P_n(m))\bigcap ((B_r\backslash B_{R(n)}) \bigcap F^c)}M(z)G_r(z,0)
\end{align*}
where $W(y),M(z)$ have the same meaning in the four situations above. We know $\forall y,z, W(y) \leq 2d, M(z) \leq 2d$. Assume that $y=x+e_k,(k \neq d)$. Then
\begin{align*}
|G_r(y,0)-G_r(x,0)|& =|G_r(x+e_k,0)-G_r(x,0)|\\
& =|a_d(|x|^{2-d}-|x+e_k|^{2-d})+O(|x|^{1-d})|\\
& \leq \frac{a_d(d-2)|x_k+\xi_k|}{|x+\xi_ke_k|^d}+O(\frac{1}{|x|^{d-1}}) \leq \frac{C_d}{|x|^{d-1}}
\end{align*}
where $\xi_k \in (0,1)$ and $C_d$ is a constant depending only on dimension $d$. Hence
\begin{align*}
& |\sum_{y\in (\partial (\bigcup_{m \in \mathbb{Z}}P_n(m))) \bigcap (B_r\backslash B_{R(n)}),x \in (\bigcup_{m \in \mathbb{Z}}P_n(m)) \bigcap (B_r\backslash B_{R(n)}),x \sim y}(G_r(y,0)-G_r(x,0))|\\
& \leq \sum_{y \in \partial P_n(0)\cap \{x_d=0\},x \in P_n(0),x \sim y}\sum_{r=0}^{\infty}|G_r(y+re_d,0)-G_r(x+re_d,0)|+\\
& \sum_{y \in \partial P_n(-1)\cap \{x_d=-1\},x \in P_n(-1),x \sim y}\sum_{r=0}^{\infty}|G_r(y-re_d,0)-G_r(x-re_d,0)|\\
& \leq \sum_{y \in \partial P_n(0)\cap \{x_d=0\},x \in P_n(0),x \sim y}\sum_{r=0}^{\infty} \frac{C_d}{|x+re_d|^{d-1}}+
 \sum_{y \in \partial P_n(-1)\cap \{x_d=-1\},x \in P_n(-1),x \sim y}\sum_{r=0}^{\infty} \frac{C_d}{|x-re_d|^{d-1}}.
\end{align*}

$\forall x \in P_n(0)$, there exists a particle which escape to infinity following the lattice line $l=\{y:y=x+ke_d, k \in \mathbb{N}\}$ after finite steps. Thus the number of the edge boundaries of $P_n(0)$ in the hyperplane $\{x: x_d=0\}$ should $\leq (2d-2)n$. For the same reason, the number of the edge boundaries of $P_n(-1)$ in the hyperplane $\{x: x_d=-1\}$ should $\leq (2d-2)n$. At the same time by Lemma 3.2 and the abelian property(Lemma 2.4), there exists a constant $\beta$ depending only on $d$ such that $u_n(x) \geq u_n^{n^{\frac{1}{d-1}}}(x)>0, \forall x \in \partial B_{\beta n^{\frac{1}{d-1}}}$. So there exists a constant $K_d$ depending only $d$ such that
\[|\sum_{y\in (\partial (\bigcup_{m \in \mathbb{R}}P_n(m))) \bigcap (B_r\backslash B_{R(n)}),x \in (\bigcup_{m \in \mathbb{R}}P_n(m)) \bigcap (B_r\backslash B_{R(n)}),x \sim y}(G_r(y,0)-G_r(x,0))| \leq K_d n^{\frac{1}{d-1}}.\]
Also
\begin{align*}
|\sum_{y \in (B_r\backslash B_{R(n)})\bigcap F^c,y\sim z}G_r(y,0)R_n(y,z)| & \leq \sum_{y \in (B_r\backslash B_{R(n)})\bigcap F^c}G_r(y,0)|\sum_{z\sim y }R_n(y,z)|\\
& \leq 8d^2 \sum_{y \in (B_r\backslash B_{R(n)})\bigcap F^c }G_r(y,0).
\end{align*}
Thus
\[|\sum_{y \in B_r\backslash B_{R(n)},y\sim z}G_r(y,0)R_n(y,z)| \leq K_dn^{\frac{1}{d-1}}+(8d^2+2d)\sum_{y \in (B_r\backslash B_{R(n)})\bigcap F^c }G_r(y,0).\]
We know that $(B_r\backslash B_{R(n)})\bigcap F^c=\partial B_{R(n)} \bigcup \partial (B_r^c)$.
$\forall r \in \mathbb{N}$, there exists a constant $C_3$ depending only on $d$ such that
\[\sum_{y \in \partial (B_r^c)}G_r(y,0)=E_0\#\{j \geq 0:X_j \in \partial (B_r^c)\}< C_3.\]
And because $\partial B_{R(n)} \subseteq S_{R(n)}$,
\[\sum_{y \in \partial B_{R(n)}}G_r(y,0) \leq \sum_{y \in S_{R(n)}}G(y,0) \leq C_4R(n)^{2-d}\cdot R(n)^{d-1}=C_4R(n).\]
We could obtain
\[|u_n^r(0)-nG_r(0,0)| \leq CR(n)+\frac{K_d}{2d} n^{\frac{1}{d-1}}+(C_4R(n)+C_3)(4d+1).\]

Since $r>R(n)$ and $\bigcup_{-h_n^- \leq r \leq h_n^+}P_n(r)\subseteq B_{R(n)}$, we have $u_n^r(0)=u_n(0)$. Let $r \rightarrow \infty$, divided by $n$ on both sides and use Lemma 3.1. Hence
\[\limsup_{n \to \infty}|\frac{u_n(0)}{n}-G(0,0)| =0.\]
So
\[\lim_{n \to \infty}\frac{n}{u_n(0)}=\alpha_d.\]
By Lemma 2.3, we could draw the conclusion of Theorem 1.2.
\end{proof}

\section{Outer estimate of rotor-router aggregation}

The estimate for rotor-router aggregation $A_n$ originates from L.Levine,Y.Peres \cite{6}. But there is a mistake in their original paper. In a personal communication with Lionel Levine, he told us a method to fix the problem. His new method could also get the outer estimate $A_n \subseteq B_{r+C^\prime r^{1-\frac{1}{d}}\log r}$ where $C^\prime$ is a constant depending only on dimension $d$. However in this problem we do not need that strong outer estimate. For a self-contained reason we follow the proof of L.Levine,Y.Peres \cite{6}. But when we handle with the iteration in the outer estimate we would not use Lionel Levine's new method and we get a relatively weaker outer estimate.
The next lemma is an unpublished result of Holroyd and Propp. Also, L.Levine,Y.Peres \cite{6} cited this lemma.
\begin{lemma}
(Holroyd and Propp) $\Gamma =(V,E)$ is a finite connected graph and $Y\subseteq Z\subseteq V$. On each site x there are $s(x)$ particles. If these particles perform independent random walks until hitting $Z$. Let $T$ be the hitting time of $Z$. $H_w(s,Y):=\sum_{x \in V}s(x)\mathbb{P}_x(X_T \in Y)$, namely the expecting number of particles stopping on $y$. If $\Gamma$ has an initial rotor configuration and these $s(x)$ particles on each site x of the graph perform rotor walk until hitting $Z$. Denote $H_r(s,Y)$ to be the number of particles on $Y$. Also let $H(x)=H_w(1_x,Y)$. We could obtain
\[|H_r(s,Y)-H_w(s,Y)| \leq \sum_{u \in V\backslash Z}\sum_{v\sim u}|H(u)-H(v)|.\]
\end{lemma}

The next two lemmas are from L.Levine,Y.Peres \cite{6}.
\begin{lemma}
$\rho \geq 1,y \in S_\rho$. For $x \in B_\rho$, let $H(x)=\mathbb{P}_x(X_T=y)$, where $T$ is the hitting time of $S_\rho$. Then
\[H(x) \leq \frac{J}{|x-y|^{d-1}}\]
where $J$ is a constant depending only on dimension $d$.
\end{lemma}

\begin{lemma}
The definition of $H(x)$ is the same as the above lemma. We could also obtain
\[\sum_{u \in B_\rho}\sum_{v \sim u}|H(u)-H(v)| \leq J^\prime \log \rho\]
where $J^\prime$ is a constant depending only on $d$.
\end{lemma}

The next estimate is weaker than the outer estimate of $A_n$ in L.Levine,Y.Peres \cite{6}. The method is also from L.Levine,Y.Peres \cite{6}.
\begin{lemma}
In $\mathbb{Z}^d$, $A_n$ is the sites occupied by particles after n particles finish their rotor-router aggregation. $r=(\frac{n}{\omega_d})^{\frac{1}{d}}$, we could obtain
\[A_n \subseteq B_{Cr(\log r)^d}\]
where $C$ is a constant depending only on dimension $d$.
\end{lemma}

\begin{proof}
For $h \geq 1$, let $\Gamma$ in Lemma 4.1 to be $B_{\rho+h+1}$ and $Z=S_{\rho+h}$. First we fix a $y \in S_{\rho+h}$ and let $Y=\{y\}$. $n$ particles in turn perform rotor walk from the origin 0 and stop until either stepping onto an unoccupied site by the previous particles or $S_{\rho+h}$. Denote $s(x)$ to be the number of the particles stopping on $x \in S_\rho$ and $H(x)=\mathbb{P}_x(X_T=y)$, where $T$ is the hitting time of $S_{\rho+h}$ for random walk. Denote $N_\rho$ to be the particles on $S_\rho$. By Lemma 4.2,
\[H_w(s,y)=\sum_{x \in S_\rho}s(x)H(x) \leq \frac{JN_\rho}{h^{d-1}}.\]
By Lemma 4.3 we obtain
\[\sum_{u \in B_{\rho+h}}\sum_{v\sim u}|H(u)-H(v)| \leq J^\prime \log(\rho+h).\]
Because of Lemma 4.1,
\[H_r(s,y) \leq \frac{JN_\rho}{h^{d-1}}+J^\prime \log (\rho+h).\]

In the original paper of L.Levine,Y.Peres \cite{6}, $J^\prime \log (\rho+h)$ in the above inequality was mistakenly witten as $J^\prime \log h$. Lionel Levine gave a fix to the problem in a personal communication. He could also get his previous outer estimate. We only need a weaker estimate and thus we will not follow his new method. But the following is similar to their original proof.

Let $\rho_0=0,\rho_{i+1}=\min \{\rho>\rho(i): N_\rho \leq \frac{N_{\rho(i)}}{2}\}$. Because of the abelian property of rotor walk, $n$ particles in turn perform rotor walk starting from the origin 0 until either entering into a site which has never been visited by the previous particles or hitting $S_{\rho(i)}$. Then let $N_{\rho(i)}$ particles on $S_{\rho(i)}$ continue to perform rotor walk until either entering into a site which has never been visited by the previous particles or hitting $S_{\rho(i)+h}$. At this time the number of the particles stop on $S_{\rho(i)+h}$ is exactly $N_{\rho(i)+h}$. Thus we could obtain
\[N_{\rho(i)+h} \leq \sum_{y \in S_{\rho(i)+h} \cap A_n}H_r(s,y).\]
Let $M_k=\# (A_n\bigcap S_k)$, so
\[M_{\rho(i)+h} \geq N_{\rho(i)+h} \frac{1}{\frac{JN_{\rho(i)}}{h^{d-1}}+J^\prime \log (\rho(i)+h)}.\]
Let $s(1)=\min \{\rho: N_\rho \leq \rho^{d-1}\log \rho\}, s(2)=\min \{\rho \geq s(1): N_\rho \leq \rho^{d-2}\log \rho\},\dots ,s(d-1)=\min \{\rho \geq s(d-2): N_\rho \leq \rho \log \rho\}$. Let $k(1)=\min \{i>0:\rho(i)<s(1)\}$. Also let $\rho(k(1)+1)=s(1)-1$. So when $0\leq i\leq k(1)$ and $1 \leq h \leq \rho(i+1)-\rho(i)-1$, $\rho(i)+h<s(1)$. Hence we could obtain
\[M_{\rho(i)+h} \geq \frac{h^{d-1}}{2J+J^\prime}.\]
Thus there exists a constant $K$ depending only on dimension $d$ such that
\[\sum_{\rho=\rho(i)+1}^{\rho(i+1)-1}M_\rho \geq K(\rho(i+1)-\rho(i))^d.\]
Let $x_i= \rho(i+1)-\rho(i)$. We know $\sum_{0 \leq \rho \leq s(1)}M_\rho \leq \omega_d r^d$. So
\[\sum_{i=0}^{k(1)}x_i^d \leq Cr^d.\]
We obtain
\[\rho(k(1)+1)=s(1)-1=\sum_{i=0}^{k(1)}x_i \leq (\sum_{i=0}^{k(1)}x_i^d)^{\frac{1}{d}}k(1)^{1-\frac{1}{d}} \leq Cr(\log r)^{1-\frac{1}{d}}.\]
The reason why the last inequality holds is that $N_{\rho(k)} \leq \frac{\omega_dr^d}{2^k}$ and there exist a constant $a>0$ such that $N_{\rho(a\log r)}=0$. Thus $s(1) \leq Cr\log r$¡£

Next would change the meaning of some symbols.
Let $\rho(0)=s(1)$ and as the method above, let $k(2)=\min \{i>0:\rho(i)<s(2)\}$, and let $\rho(k(2)+1)=s(2)-1$. So when $0\leq i\leq k(2), 1 \leq h \leq \rho(i+1)-\rho(i)-1$, we have $\rho(i)+h<s(2)$. We can obtain
\[M_{\rho(i)+h} \geq \frac{h^{d-2}}{2J+J^\prime}.\]
Thus
\[\sum_{\rho=\rho(i)+1}^{\rho(i+1)-1}M_\rho \geq C(\rho(i+1)-\rho(i))^{d-1}\]
Similarly let $x_i=\rho(i+1)-\rho(i)$ and we could know
\[\sum_{i=0}^{k(1)}x_i^{d-1} \leq s(1)^{d-1}\log s(1) \leq Cr^{d-1}(\log r)^{d-1+\frac{1}{d}}.\]
So
\[s(2)-1-s(1)=\sum_{i=0}^{k(2)}x_i \leq (\sum_{i=0}^{k(2)}x_i^{d-1})^{\frac{1}{d-1}}k(2)^{1-\frac{1}{d-1} }\leq Cr(\log r)^{2-\frac{1}{d}} \leq Cr(\log r)^2\]
We obtain $s(2) \leq Cr(\log r)^2$. Similarly keep using this method for another $d-3$ times we obtain $s(d-1) \leq Cr(\log r)^{d-1}$. But $N_{s(d-1)} \leq s(d-1)\log [s(d-1)] \leq Cr(\log r)^d$. So we could know $A_n\subseteq B_{Cr(\log r)^d}$.
\end{proof}

\section{Estimates for height}

The next two sections are devoted to prove Lemma 3.1.

\begin{lemma}
When $d \geq 3$ and the initial configuration is $\rho_0$, $n$ particles in turn perform rotor walk starting from 0 until escaping to infinity. Then there exists a constant $C$ depending only on dimension $d$ such that
\[h_n^+ \leq Cn^{\frac{2}{3}}(\log n)^2, h_n^- \leq Cn^{\frac{2}{3}}(\log n)^2.\]
\end{lemma}

\begin{proof}
Let $h_n^{d+},h_n^{d+},I^{(d)}(\rho,n), u_n^{(d)}(0),x_i^{(d)}$ to be the corresponding $h_n^+,h_n^-,I(\rho,n), u_n(0), x_i$ in $d$-dimensional case.
Below is an extension of the definition of $h_n^+, h_n^-$. If there is a map $R:\mathbb{Z}^d \rightarrow \mathbb{N}, R(x)$ denote to be the number of particles on $x$, while $\rho$ represents the rotor configuration corresponding to the particle distribution $R$. $(R,\rho)$ indicates a state of the rotor walk. If we let $(R,\rho)$ perform rotor walk until escaping to infinity, denote $h^+(R,\rho)$ to be the maximal $d$th-dimensional coordinate of the sites from which have been exited by the particles at least twice. Similarly, $-h^-(R,\rho)$ is denoted to be the minimum $d$th-dimensional coordinate of the sites which have been exited by the particles at least twice.

Firstly we prove the dimension $d=3$ case. When initial rotor configuration is $\rho_0$, $n$ particles perform rotor-router aggregation. The sites occupied by particles are denoted as $A_n$. By Lemma 4.4, there exists a constant $K$ such that
\[A_n \subseteq B_{Kn^{\frac{1}{3}}\log n}.\]
Denote $\rho^\prime$ to be the rotor configuration after $n$ particles starting from the origin finishes their rotor-router aggregation. If we put a particle on each site of $B_{Kn^{\frac{1}{3}}\log n}$ without change the rotor configuration $\rho^ \prime$, denote the corresponding particle distribution to be
\[G: \mathbb{Z}^d \rightarrow \mathbb{N}, G(x)=1_\{x \in B_{Kn^{\frac{1}{3}}\log n}\}.\]

Let $(n1_{\{x=0\}}, \rho_0)$ perform rotor walk until escaping to infinity, we can deduce like the beginning of Theorem 1.1. We obtain that for $x=(x_1, x_2, x_3)$ that has been visited by the particles, $\forall 1 \leq i \leq 2, |x_i| \leq n$. Also, $h_n^+ \leq n$ and $h_n^- \leq n$. If we let $(1_{\{x \in A_n\}}, \rho^\prime)$ perform rotor walk until escaping to infinity, similarly, for $x=(x_1, x_2, x_3)$ that has been visited in the process, we could choose a constant $C^\prime$, $\forall 1 \leq i \leq 2, |x_i|<CK^3n(\log n)^3+Kn^{\frac{1}{3}}\log n< C^ \prime n(\log n)^3$. Also, similar to the proof of Lemma 2.2, $h^+(1_{\{x \in A_n\}}, \rho^\prime)< CK^3n(\log n)^3+Kn^{\frac{1}{3}}\log n< C^ \prime n(\log n)^3$ and $h^-(1_{\{x \in A_n\}}, \rho^\prime)<C^ \prime n(\log n)^3$.

If $h_n^+ \leq Kn^{\frac{1}{3}}\log n$, the conclusion for $d=3$ case follows. Else, we consider $n$ particles perform rotor walk from 0 until hitting
\[D_1:=\{{x:x_3=h_n^+}\} \bigcup \{{x:x_3=-C^ \prime n(\log n)^3}\} \bigcup (\bigcup_{i=1}^2\{x_i= \pm C^ \prime n(\log n)^3\}).\]
There are at least two particles staying on a single site of hyperplane $\{{x:x_3=h_n^+}\}$.
Another way to realize that is to let the $n$ particles on the origin 0 perform rotor-router aggregation. And then continue to let $(1_{\{x \in A_n\}},\rho^ \prime)$ perform rotor walk until hitting $D_1$. By the definition of $h^+(1_{\{x \in A_n\}},\rho^ \prime)$ and the abelian property, we obtain $h_n^+ \leq h^+(1_{\{x \in A_n\}},\rho^ \prime)$.

Let $(G,\rho^ \prime)$ perform rotor walk until hitting
\begin{align*}
D_2:=  & \{{x:x_3=h^+(1_{\{x \in A_n\}},\rho^ \prime)}\} \bigcup \{{x:x_3=-C^ \prime n(\log n)^3}\}\\
 & \bigcup (\bigcup_{i=1}^2\{x_i= \pm C^ \prime n(\log n)^3\}).
\end{align*}
Another way to realize this is to let $(1_{\{x \in A_n\}} , \rho^ \prime)$ perform rotor walk until hitting $D_2$. There are at least two particles staying on a single site of hyperplane $\{{x:x_3=h^+(1_{\{x \in A_n\}}}, \rho^\prime\}$. And then let the particles in $B_{Kn^{\frac{1}{3}}\log n} \backslash A_n$ continue to perform rotor walk until hitting $D_2$. By the abelian property we know $h^+(1_{\{x \in A_n\}},\rho^ \prime) \leq h^+(G,\rho^ \prime)$.

Now consider the rotor walk state $(6G,\rho_0)$. Let $(6G,\rho_0)$ perform rotor walk until hitting
\[D_3:=\{{x:x_3=h^+(G,\rho^ \prime)}\} \bigcup \{{x:x_3=-C^ \prime n(\log n)^3}\} \bigcup (\bigcup_{i=1}^2\{x_i= \pm C^ \prime n(\log n)^3\}).\]
We could perform the rotor walk in another way. $\forall y \in B_{Kn^{\frac{1}{3}}\log n}$, we do the follow operations to $y$. If $m(y)^{(k)}(\rho_0(y))=\rho^ \prime (y)$, let $k$ particles on $y$ perform one-step rotor walk. When these operations finish, denote the rotor walk state to be $(U,\rho^ \prime)$. Obviously we have $U(y) \geq 1, \forall y \in B_{Kn^{\frac{1}{3}}\log n} $. Similar to the arguments above, let $(G,\rho^ \prime)$ perform rotor walk until hitting $D_3$. Then let the rest particles in $B_{Kn^{\frac{1}{3}}\log n}$ continue to perform rotor walk until hitting $D_3$. By the abelian property,  $h^+(G,\rho^ \prime) \leq h^+(6G,\rho_0)$.

Let $(6G,\rho_0)$ perform rotor walk until escaping to infinity. The same as the previous case, we know $\forall 1 \leq i \leq 2, |x_i|<6CK^3n(\log n)^3+Kn^{\frac{1}{3}}\log n< 6C^ \prime n(\log n)^3$, where $x=(x_1, x_2, x_3)$ is a site which has been visited during the process.

Next we would construct rotor walk on $\mathbb{Z}/ (12C^ \prime n(\log n)^3\mathbb{Z}) \times \mathbb{Z}^2$. Let $\rho_0^*=+e_3$ if $x_3 \geq 0$, $\rho_0^*=-e_3$ if $x_3 < 0$.
\[ R(x):=\begin{cases}
1 & x \in I:=\{(x_1,x_2,x_3) \in \mathbb{Z}/ (12C^ \prime n(\log n)^3\mathbb{Z}) \times \mathbb{Z}^2: x_2^2+x_3^2< K^2n^{\frac{2}{3}}(\log n)^2\}\\
0 & x \in (\mathbb{Z}/ (12C^ \prime n(\log n)^3\mathbb{Z}) \times \mathbb{Z}^2)\backslash I
\end{cases} .\]

Denote $h_0^+(. , .), h_0^-(. , .)$, $G^*$ to be the corresponding definition of $h^+(. , .), h^-(. , .)$, $G$ in \\
$\mathbb{Z}/ (12C^ \prime n(\log n)^3\mathbb{Z}) \times \mathbb{Z}^2$. Similar to the above arguments, we know if let $(6R,\rho_0^*)$ performs rotor walk until escaping to infinity. For a site $x$ that has been visited during the process, there exists a constant $C^ {\prime\prime}$ such that $h_0^+(6R, \rho_0^*) < C^ {\prime\prime} n^{\frac{5}{3}}(\log n)^5, h_0^-(6R, \rho_0^*) > -C^ {\prime\prime} n^{\frac{5}{3}}(\log n)^5, |x_2|< C^ {\prime\prime} n^{\frac{5}{3}}(\log n)^5$. Let $(6R,\rho_0^*)$ perform rotor walk until hitting
\begin{align*}
D_4: & = \{x \in \mathbb{Z}/ (12C^ \prime n(\log n)^3\mathbb{Z}) \times \mathbb{Z}^2: x_3= h_0^+(6G^*,\rho_0^*)\}\\
 & \bigcup \{x \in \mathbb{Z}/ (12C^ \prime n(\log n)^3\mathbb{Z}) \times \mathbb{Z}^2: x_3= -C^ {\prime\prime} n^{\frac{5}{3}}(\log n)^5\} \bigcup
 \{|x_2|=C^ {\prime\prime} n^{\frac{5}{3}}(\log n)^5\}.
\end{align*}
Another way is to let $(6G^*, \rho_0^*)$ perform rotor walk until hitting $D_4$ first and then let the rest of the particles continue to perform rotor walk. By abelian property, we obtain $h_0^+(6G^*,\rho_0^*) \leq h_0^+(6R,\rho_0^*)$.

By the method we choose $C^ \prime$, $h_0^+(6G^*,\rho_0^*)=h^+(6G,\rho_0)$.

Now let $(6R,\rho_0^*)$ perform rotor walk until hitting
\begin{align*}
D_5: & = \{x \in \mathbb{Z}/ (12C^ \prime n(\log n)^3\mathbb{Z}) \times \mathbb{Z}^2: x_3= h_0^+(6R,\rho_0^*)\}\\
 & \bigcup \{x \in \mathbb{Z}/ (12C^ \prime n(\log n)^3\mathbb{Z}) \times \mathbb{Z}^2: x_3= -C^ {\prime\prime} n^{\frac{5}{3}}(\log n)^5\} \bigcup
 \{|x_2|=C^ {\prime\prime} n^{\frac{5}{3}}(\log n)^5\}.
\end{align*}
Another way is to let $\{(x,x_2,x_3):x \in \mathbb{Z}/ (2C^ \prime n(\log n)^3\mathbb{Z})\}$ perform one-step rotor walk simultaneously and when we regard $\{(x,x_2,x_3):x \in \mathbb{Z}/ (2C^ \prime n(\log n)^3\mathbb{Z})\}$ as a 2-dimensional site $(x_2,x_3)$, it is the same with a one-step rotor walk of $(x_2, x_3)$ in $\mathbb{Z}^2$. We stop when particles hit $D_5$.
We know that the previous process is the same as $(61_{\{|x|<Kn^{\frac{1}{3}}\log n\}},\rho_0)$ performing rotor walk in $\mathbb{Z}^2$ until hitting
\begin{align*}
D_6: & =  \{x \in  \mathbb{Z}^2: x_3= h_0^+(6R,\rho_0^*)\}
  \bigcup \{x \in \mathbb{Z}^2: x_3= -C^ {\prime\prime} n^{\frac{5}{3}}(\log n)^5\} \\
  & \bigcup \{|x_2|=C^ {\prime\prime} n^{\frac{5}{3}}(\log n)^5\}.
\end{align*}
Similar to Lemma 2.2, we know there exists a constant such that
\[h_0^+(6R,\rho_0^*) \leq h^{2+}(61_{\{|x|<Kn^{\frac{1}{3}}\log n\}},\rho_0) \leq C(Kn^{\frac{1}{3}}\log n)^2+Kn^{\frac{1}{3}}\log n \leq Cn^{\frac{2}{3}}(\log n)^2.\]

Thus $h_n^+ \leq Cn^{\frac{2}{3}}(\log n)^2$. The same method could be used to prove $h_n^- \leq Cn^{\frac{2}{3}}(\log n)^2$.

When $d>3$ and the initial configuration is $\rho_0$, $n$ particles in turn perform rotor walk. For particle $x$ that has been visited during the process, we have $|x_1|<n+1$. Now construct rotor walk in $\mathbb{Z}/(2(n+1)\mathbb{Z})\times \mathbb{Z}^{d-1}$. $\forall x \in \mathbb{Z}/(2(n+1)\mathbb{Z})\times \mathbb{Z}^{d-1}$, denote $\rho_0^*(x)=+e_d$ if $x_d \geq 0$, $\rho_0^*(x)=-e_d$ if $x_d <0$. Let
\[S(x):=\begin{cases}
1 & x \in L:=\{(x_1,0,\dots,0): x_1 \in \mathbb{Z}/(2(n+1)\mathbb{Z} )\}\\
0 & x \in (\mathbb{Z}/(2(n+1)\mathbb{Z})\times \mathbb{Z}^{d-1})\backslash L
\end{cases} .\]

Denote $h_0^+(.,.), h_0^-(.,.)$ to be the corresponding definition in $\mathbb{Z}/(2(n+1)\mathbb{Z})\times \mathbb{Z}^{d-1}$. Let $(nS,\rho_0^*)$ perform rotor walk until escaping to infinity. For site $x$ that has been visited by these particles, we know $\forall 1 \leq i \leq d-1$, $|x_i|<Cn^2$ and $h_0^+(nS,\rho_0^*)<Cn^2$, $h_0^+(nS,\rho_0^*)>-Cn^2$ where $C$ is a constant. In $\mathbb{Z}/(2(n+1)\mathbb{Z})\times \mathbb{Z}^{d-1}$ $(nS,\rho_0^*)$ perform rotor walk until hitting
\begin{align*}
D_7: & =\{x \in \mathbb{Z}/(2(n+1)\mathbb{Z})\times \mathbb{Z}^{d-1} : x_d=h_0^+(n1_{\{x=0\}},\rho_0^*)\}\\
& \bigcup \{x \in \mathbb{Z}/(2(n+1)\mathbb{Z})\times \mathbb{Z}^{d-1} : x_d=-Cn^2.\} \bigcup (\bigcup_{i=2}^{d-1}\{|x_i|=Cn^2\}).
\end{align*}
Another way is to let $(n1_{\{x=0\}},\rho_0^*)$ perform rotor walk until hitting $D_7$ first and then let the rest of the particles perform rotor walk until hitting $D_7$. By abelian property we know $h_0^+(n1_{\{x=0\}},\rho_0^*) \leq h_0^+(nS,\rho_0^*) $¡£

The same as the previous case we know $h_0^+(nS,\rho_0^*) \leq h_n^{(d-1)+}$ and $h_0^+(n1_{\{x=0\}},\rho_0^*) = h_n^{d+}$. Hence $h_n^{d+} \leq h_n^{(d-1)+}$.

Our assumption (1) for cyclical order is that $(\eta(e_{d-1})-\eta(-e_d))(\eta(-e_{d-1})-\eta(-e_d))<0$. Notice the previous $d=3$ case need the two 2th-dimensional directions seperate $e_d$ and $-e_d$ like Lemma 2.1 and Lemma 2.2. So we can use the above method $(d-3)$ times(curl the first $d-3$ dimension coordinates) and thus we know when $d \geq 3$, we have $h_n^+ \leq Cn^{\frac{2}{3}}(\log n)^2$ where $C$ is a constant depending only on dimension $d$. The same method could be used to prove $h_n^- \leq Cn^{\frac{2}{3}}(\log n)^2$.
\end{proof}

\section{Estimates for breadth}

This section we will give an estimate for breadth under assumption (1) we mentioned in Introduction. The intuition of the estimates for breadth is that when we regard the lattice line $l_x={\{y: y=x+ke_d, k \in \mathbb{Z}\}}$ in $\mathbb{Z}^d$ as a single site in $\mathbb{Z}^{d-1}$, $n$ particles in turn performing rotor walk in $\mathbb{Z}^d$ from the origin 0 until escaping to infinity is similar to $n$ particles' rotor-router aggregation. Because once a particle reaches $l_x$, it would escapes directly to infinity follow the direction either $e_d$ or $-e_d$ and we could regard it as getting trapped in a site in $\mathbb{Z}^{d-1}$. The following proof is similar to the outer estimate for rotor-router aggregation in L.Levine,Y.Peres \cite{6}.

For $x=(x_1, \dots, x_d) \in \mathbb{Z}^d$, let $\tilde{x}=(x_1,x_2,\dots ,x_{d-1}), Cylinder_r=\{x \in \mathbb{Z}^d: |\tilde{x}|<r\}, T_r=\{x \in \mathbb{Z}^d:r \leq |\tilde{x}| <r+1\}$. When the initial rotor configuration is $\rho_0$, $n$ particles in turn perform rotor walk from the origin 0 in $\mathbb{Z}^d$. If these particles stop until hitting $T_r \bigcup \{\infty\}$, denote the number of the particles on $T_r$ to be $N_r^*$. For $x \in T_r$, denote the number of the particles on $x$ to be $s(x)$ and on $l_{\tilde {x}}=\{(\tilde{x},y):y \in \mathbb{Z}\}$ to be $\tilde s(\tilde{x})$. And then we let the $N_r^*$ particles on $T_r$ continue to perform rotor walk until hitting $T_{r+h}\bigcup \{\infty\}$. For $y \in T_{r+h}$, denote $H_r^*(s,\tilde{y})$ to be the number of particles on $l_{\tilde{y}}$ where $s$ represents the distribution of the $N_\rho^*$ particles on $T_r$. By the abelian property We know the number of the particles on $T_{r+h}$ is $N_{r+h}^*$.

Denote $B_{r}^{(d-1)}$ and $S_{r}^{(d-1)}$ to be the $d-1$ dimensional Euclidian lattice ball and $d-1$ dimensional Euclidian lattice sphere. $y \in S_{r+h}^{(d-1)}$ and for all $x \in S_r^{(d-1)}$, $p(x)$ particles stay on $x$. These particles located in $S_r^{(d-1)}$ begin to perform independent random walks until hitting $S_{r+h}^{(d-1)}$. Denote $H_w(p,y)$ to be the expecting number of particles stopping on $y$. Also, we denote $H_w(1_{x},y)$ to be $H(x)$.

\begin{lemma}
When the initial rotor configuration is $\rho_0$, $n$ particles in $\mathbb{Z}^d$ perform rotor walk starting from 0 until hitting $T_r \bigcup \{\infty\}$. $N_r^*$ particles stay on $T_r$ and $s(x)$ particles stay on $x$. Let these $N_r^*$ particles continue to perform rotor walk until hitting $T_{r+h} \bigcup \{\infty\}$. Then there exists a constant $C$ depending only on dimension $d$ such that $\forall n,r,h \geq 1, \forall y \in T_{r+h}$£¬
\[H_r^*(s,\tilde{y}) \leq H_w(\tilde {s},\tilde{y})+Cn^{\frac{2}{3}}(\log n)^2\sum_{\tilde{u} \in B_{r+h}^{(d-1)}}\sum_{\tilde{v} \sim \tilde{u}}|H(\tilde{u})-H(\tilde{v})|\]
\end{lemma}

\begin{proof}
By Lemma 5.1, we know there exists a constant $C, h_n^+<Cn^{\frac{2}{3}}(\log n)^2,h_n^-<Cn^{\frac{2}{3}}(\log n)^2$.
Let
\[L_n(r+h)=Cylinder_{r+h}\bigcap \{x: |x_d| \leq Cn^{\frac{2}{3}}(\log n)^2\}.\]
When the initial configuration is $\rho_0$, $n$ particles perform rotor walk until hitting
\[T_r\bigcup \{x:|x_d|=Cn^{\frac{2}{3}}(\log n)^2\}.\]
Denote the rotor configuration in $L_n(r+h)$ to be $\zeta_0$ when the above rotor walk finishes.

Next we offer weight to each edge of $L_n(r+h)$. $\forall z \in L_n(r+h)$, let $w(z,z+\zeta_0(z))=0$ and
\[w(z,z+m^{(k+1)}(\zeta_0(z)))=H(\tilde{z})-H(\widetilde{z+m^{(k)}(\zeta_0(z))})+w(z,z+m^{(k)}(\zeta_0(z))).\]
By $\sum_{k=1}^{2d}H(\widetilde{z+m^{(k)}(\zeta_0(z))})=2dH(\tilde{z})$ we know the definition of the edge weight is well-defined. For all $x \in T_r$, we give weight $H(\tilde{x})$ to every particle located in $l_{\tilde{x}}$.

Now we let the $N_r^*$ particles located in $T_r$ begin to perform rotor walk until hitting
\[D_8:=T_{r+h} \bigcup \{x:|x_d|=Cn^{\frac{2}{3}}(\log n)^2\}.\]
Notice that when particles hit $\{x:|x_d|=Cn^{\frac{2}{3}}(\log n)^2\}$, it means that the particles would escape to $\infty$ directly without influencing the rotor configuration in $L_n(r+h)$. Denote $(U,\zeta)$ to be a rotor walk state during the above process where $U: L_n(r+h) \rightarrow \mathbb{N},U(x)=k$, meaning that $k$ particles stays on $x$. By the definition of the edge weight and the particle weight, we know during the process of rotor walk
\[\sum_{x \in L_n(r,r+h)}U(x)H(\tilde{x})+\sum_{x \in L_n(r,r+h)}w(x,x+\zeta(x)) \equiv const.\]
In the beginning the sum of the particle weights is $H_w(\tilde{s},\tilde{y})$. After all these particles hit $D_8$, the sum of the particle weights is
\[H_r^*(s,\tilde{y})+\sum_{x \in \{x:|x_d|=Cn^{\frac{2}{3}}(\log n)^2\}}U_f(x)H(\tilde{x})\]
where $U_f$ is the final rotor walk distribution. The difference of the two terms is the change of the edge weights, which is controlled by $2Cn^{\frac{2}{3}}(\log n)^2\sum_{\tilde{u} \in B_{r+h}^{(d-1)}}\sum_{\tilde{v} \sim \tilde{u}}|H(\tilde{u})-H(\tilde{v})|$. Note $U_f \geq 0, H(\tilde{.}) \geq 0$, hence the desired result follows.
\end{proof}

Next we give a proof of Lemma 3.1.

\begin{proof}
By Lemma 4.2 we know $H_w(s,\tilde{y}) \leq \frac{JN_r^*}{h^{d-2}}$ and by Lemma 4.3 we obtain $\sum_{\tilde{u} \in B_{r+h}^{(d-1)}}\sum_{\tilde{v} \sim \tilde{u}}|H(\tilde{u})-H(\tilde{v})| \leq J^\prime \log (r+h)$. Hence
\[H_r^*(s,\tilde{y}) \leq \frac{JN_r^*}{h^{d-2}}+Cn^{\frac{2}{3}}(\log n)^2\log (r+h).\]
When the initial rotor configuration is $\rho_0$, $n$ particles in turn perform rotor walk until escaping to infinity. Denote $M_{r+h}^*$ to be the number of the lattice line $l_{\tilde{y}}$ visited during the above process in $T_{r+h}$. By the abelian property of rotor walk
\[N_{r+h}^*=\sum_{\tilde{y} \in B_{r+h}^{(d-1)}}H_r^*(s,\tilde{y}) \leq M_{r+h}^*(\frac{JN_r^*}{h^{d-2}}+Cn^{\frac{2}{3}}(\log n)^2\log (r+h)).\]
Hence
\[M_{r+h}^* \geq \frac{N_{r+h}^*}{\frac{JN_r^*}{h^{d-2}}+Cn^{\frac{2}{3}}(\log n)^2\log (r+h)}.\]

Let $s=\min \{t \in \mathbb{N}: N_t^* \leq t\log t\}, \rho(0)=0,\rho(i+1)= \min \{t>\rho(i): N_{\rho(i+1)}^* \leq \frac{N_{\rho(i)}^*}{2}\},k= \max \{r \in \mathbb{N}:\rho(r)<s\}$. The same as the proof of rotor-router aggregation, $\exists a>0$ such that $k< a\log n$. Let $\rho(k+1)=s-1$ and we know
\[\sum_{h=1}^{\rho(i+1)-\rho(i)-1}M_{r+h}^* \geq \frac{(\rho(i+1)-\rho(i))^2}{J+Cn^{\frac{2}{3}}(\log n)^2}.\]
Let $x_i=\rho(i+1)-\rho(i)$. By Lemma 2.1, we know $\sum_{h \geq 0}M_h^* \leq n$. Thus
\[\sum_{i=1}^kx_i^2 \leq (J+Cn^{\frac{2}{3}}(\log n)^2)n .\]
Hence
\[s-1=\sum_{i=1}^kx_i \leq k^{\frac{1}{2}}(\sum_{i=1}^kx_i^2)^{\frac{1}{2}} \leq C (\log n)^{\frac{3}{2}}n^{\frac{5}{6}}.\]
For a site $x=(x_1, \dots, x_{d-1}, x_d)$
that has been visited during the process,
\[|x_i| \leq s+s\log s \leq C(\log n)^{\frac{5}{2}}n^{\frac{5}{6}}\]
where $1 \leq i \leq d-1$.
Also, we know $h_n^+ \leq Cn^{\frac{2}{3}}(\log n)^2, h_n^- \leq Cn^{\frac{2}{3}}(\log n)^2$. Thus there exists a constant $K_d$ such that
$R(n)=K_d(\log n)^{\frac{5}{2}}n^{\frac{5}{6}}$. We have
\[\bigcup_{-h_n^- \leq r \leq h_n^+}P_n(r) \subseteq B_{R(n)}\]
and
\[\lim_{n \to \infty}\frac{R(n)}{n}=0.\]
\end{proof}

\begin{remark}
Another way to give estimate for $|x_i|$ where $1 \leq i \leq d-1$ is based on a modification of Lemma 2.1. We could prove even when the rotor cyclical order $m$ is arbitrary Lemma 2.1 is valid and the structure of rotor walk is the same. The definitions of $h_n^+$ and $h_n^-$ still make sense. Let $n$ particles in turn perform rotor walk until escaping to infinity and define $x_i^+(n)$, $-x_i^-(n)$ to be the maximal and minimum $i$th dimensional coordinate of a site $x$, respectively. Although we cannot know whether $h_n^+ \lesssim n$ and $h_n^- \lesssim n$ is sill valid, $x_i^+(n) \leq n$ and $x_i^-(n) \leq n$ for all $1 \leq i \leq d-1$ still remain correct. Next we could use the method in Lemma 5.1. We need only to extend the definitions of $x_i^+(n)$, $x_i^-(n)$ to $x_i^+(R,\rho)$ and $x_i^-(R,\rho)$, just as an extension of $h_n^+$ to $h^+(R, \rho)$ in Lemma 5.1. In $\mathbb{Z}^3$ if we want to estimate $x_1^+(n)$, similar to Lemma 5.1 we need to curl the 2th-dimensional direction. Note in the proof in order to use the abelian property we need to confine the rotor walk into a box but we do not need to know the exact length of the edges of the box except the 1st-dimensional direction edges. Hence we could obtain $x_1^+(n) \leq Cn^{\frac{2}{3}}(\log n)^2$ and $x_1^-(n) \leq Cn^{\frac{2}{3}}(\log n)^2$. Everything goes through and then we know $x_i^+(n) \leq Cn^{\frac{2}{3}}(\log n)^2$ and $x_i^-(n) \leq Cn^{\frac{2}{3}}(\log n)^2$ for $1 \leq i \leq d-1$. This is even stronger than the above estimate. As a result we come to the conclusion of Lemma 3.1.
\end{remark}

\section*{Acknowledgement}

The author would like to thank Prof. Jiangang Ying for his encouragements and supports.


\begin{thebibliography}{99}
\bibitem{1} {\sc Omer Angel and Alexander E. Holroyd}. {\em Rotor walks on general trees}. SIAM J. Discrete Math., 21:423-446,2011. 
\bibitem{2} {\sc Omer Angel and Alexander E. Holroyd}. {\em Recurrent rotor-router configurations}. 2011. 
\bibitem{3}  {\sc Alexander E. Holroyd, Lionel Levine, Karola Meszaros, Yuval Peres, James Propp, and David B. Wilson}.{\em Chip-firing and rotor-routing on directed graphs}. In and out of equilibrium. 2, volume 60 of Progr. Probab., pages 331-364. Birkhauser,Basel. 2008. 
\bibitem{4} {\sc Omer Angel and Alexander E. Holroyd}. {\em Rotor walks and Markov chains}. Algorithmic Probablility and Combinatiorics, 520:105-126,2010. 
\bibitem{5} {\sc G. F. Lawler}. {\em Intersections of random walks}. Birkhauser, Boston. 1996.
\bibitem{6} {\sc Lionel Levine and Yuval Peres}. {\em Strong spherical asymptotics for rotor-router aggregation and the divisible sandpile}.
 Potential Analysis, 30:1-27, 2009. 
\bibitem{7} {\sc Laura Florescu, Shirshendu Ganguly, Lionel Levine, Yuval Peres}. {\em Escape rates for rotor walks in $\mathbb{Z}^d$}. SIAM J. Discrete Math., 28(1):323-334, 2014. 
\bibitem{8} {\sc V.B. Priezzhev, Deepak Dhar, Abhishek Dhar, and Supriya Krishnamurthy}. {\em Eulerian walkers as a model of self-organised criticality}. Phys. Rev. Lett., 77:5079-5082, 1996. 
\bibitem{9} {\sc James Propp}. {\em Random walk and random aggregation}.derandomized. 2003.
\bibitem{10} {\sc Oded Schramm}. {\em Personal communication}. 2003.
\bibitem{11} {\sc Israel A. Wagner, Michael Lindenbaum, and Alfred M. Bruckstein}. {\em Smell as a computational resource (a lesson we can learn from the ant)}. 4th Israel Symposium on Theory of Computing and Systems (ISTCS '96), pages 219-230, 1996.
\bibitem{12} {\sc Tulasi Ram Reddy A}. {\em A recurrent rotor-router configuration in $\mathbb{Z}^3$}. 2010. 

\end{thebibliography}
\end{document}